\documentclass[12pt,reqno]{amsart}

\usepackage{tikz}
\usepackage{xcolor}

\usepackage{latexsym,amsmath,amssymb,epic}
\usepackage{latexsym,amssymb,epic}
\usepackage{amsmath,amsthm}
\usepackage{color}
\usepackage{nicematrix}

\usepackage[left=3.75cm,right=3.75cm]{geometry}
\usepackage{etoolbox}
\usepackage{enumitem}
\usepackage[noadjust]{cite}
\usepackage[pagebackref]{hyperref}

\usepackage{hyperref}
\hypersetup{
	colorlinks=true, 
	linktoc=all, 
	linkcolor=blue} 

\allowdisplaybreaks[4]

\numberwithin{equation}{section}

\theoremstyle{theorem}
\newtheorem{theorem}{Theorem}
\newtheorem*{theorem*}{Theorem}

\newtheorem{lemma}[theorem]{Lemma}

\theoremstyle{definition}

\newtheorem*{example*}{Example}
\newtheorem{conjecture}{Conjecture}
\newtheorem*{conjecture*}{Conjecture}

\newtheorem*{remark*}{Remark}
\newtheorem*{remarks*}{Remarks}

\makeatletter
\patchcmd{\section}{\scshape}{\bfseries\boldmath}{}{}
\patchcmd{\subsection}{\bfseries}{\bfseries\boldmath}{}{}
\renewcommand{\@secnumfont}{\bfseries}
\makeatother

\begin{document}
	
\title[Arithmetic Properties for Schur-Type Overpartitions]{Elementary Proofs of Arithmetic Properties for Schur-Type Overpartitions Modulo Small Powers of $2$}

\author[S. Chern]{Shane Chern}
\address[S. Chern]{Department of Mathematics and Statistics, Dalhousie University, Halifax, NS, B3H 4R2, Canada}
\email{chenxiaohang92@gmail.com}
	
\author[R. da Silva]{Robson da Silva}
\address[R. da Silva]{Universidade Federal de S\~ao Paulo, S\~ao Jos\'e dos Campos, SP 12247--014, Brazil}
\email{silva.robson@unifesp.br}

\author[J. A. Sellers]{James A. Sellers}
\address[J. A. Sellers]{Department of Mathematics and Statistics, University of Minnesota Duluth, Duluth, MN 55812, USA}
\email{jsellers@d.umn.edu}

	
	
\maketitle
\begin{abstract}
In 2022, Broudy and Lovejoy extensively studied the function $S(n)$ which counts the number of overpartitions of \emph{Schur-type}.  In particular, they proved a number of congruences satisfied by $S(n)$ modulo $2$, $4$, and $5$.  In this work, we extend their list of arithmetic properties satisfied by $S(n)$ by focusing on moduli which are small powers of 2.  In particular, we prove the following infinite family of Ramanujan-like congruences:   
\textit{For all $\alpha\geq 0$ and $n\geq 0$},
$$
S\left(2^{5+2\alpha}n+\left(2^{5+2\alpha}-\frac{2^{2+2\alpha}-1}{3}\right)\right)\equiv 0 \pmod{16}.
$$
All of the proof techniques used herein are elementary, relying on classical $q$-series identities and generating function manipulations as well as the parameterization work popularized by Alaca, Alaca, and Williams.  
\end{abstract}

\section{Introduction}  

In 2022, Broudy and Lovejoy \cite{BL} defined $S(n)$ to be the number of overpartitions $(\lambda_1, \lambda_2, \dots, \lambda_s)$ of $n$ where only parts divisible by $3$ may occur non-overlined, with the conditions that 
\begin{itemize}[leftmargin=*,align=left,itemsep=6pt]
\item{} the smallest part is $\overline{1}$, $\overline{2}$, $\overline{3}$ or $6$ modulo $6$; 
\item{} for $u, v\in \left\{  \overline{1}, \overline{2}, \overline{3}, 3\right\}$, if $\lambda_i\equiv u \pmod{3}$ and $\lambda_{i+1}\equiv v \pmod{3}$, then $$\lambda_i - \lambda_{i+1} \geq \overline{A}_{3,1}(u,v);$$ and 
\item{} for $u, v\in \left\{  \overline{1}, \overline{2}, \overline{3}, 3\right\}$, if $\lambda_i\equiv u \pmod{3}$ and $\lambda_{i+1}\equiv v \pmod{3}$, then $$\lambda_i - \lambda_{i+1} \equiv \overline{A}_{3,1}(u,v) \pmod{6},$$
where $\overline{A}_{3,1}$ is defined by the matrix 
$$\overline{A}_{3,1} \ :=\ \  \begin{pNiceMatrix}[first-row,first-col,nullify-dots]
& \overline{1} & \overline{2} & \overline{3} & 3 \\
\overline{1} & 3 & 2 & 4 & 1  \\
\overline{2} & 4 & 3 & 5 & 2  \\
\overline{3} & 5 & 4 & 6 & 3 \\
3 & 2 & 1 & 3 & 0 
\end{pNiceMatrix}.$$
\end{itemize}
Indeed, this function $S(n)$ is a special case of a more general result of Bringmann, Lovejoy, and Mahlburg \cite{BLM}; see also the work of Sang and Shi \cite{SS} where the work of Bringmann, Lovejoy, and Mahlburg is generalized.  

Broudy and Lovejoy say that the overpartitions counted by $S(n)$ are of \emph{Schur-type} to reflect their connection with Schur's celebrated partition theorem \cite{Sch1926}, and they note that the generating function for $S(n)$ is given by 
\begin{equation*}
\sum_{n=0}^{\infty} S(n)q^n  = \frac{1}{(q;q^6)_\infty(q^5;q^6)_\infty(q^6;q^6)_\infty}  
\end{equation*}
where $(a;q)_\infty := (1-a)(1-aq)(1-aq^{2})\cdots$
is the usual \emph{$q$-Pochhammer symbol}. 
The primary objective of Broudy and Lovejoy is to consider various arithmetic properties satisfied by $S(n)$, and they accomplish this goal using a combination of tools from the theory of modular forms as well as elementary generating function manipulations.  In particular, they prove a number of arithmetic properties satisfied by $S(n)$ modulo $2$, $4$, and $5$ in various arithmetic progressions.  

In their concluding remarks, Broudy and Lovejoy share the following:  
\begin{quotation}
Using \textsf{Mathematica}, we computed the values of $S(n)$ up to $n=40,000$ and searched for (likely) congruences of the form 
$$S(An+B)\equiv 0 \pmod{M}$$
for $M\in [2,20] \cup \{32,64\}$ and $A\leq 5,000$.  The results in this paper do not cover all of the congruences we detected.  $\dots$  We also observed some apparently sporadic congruences modulo $8$, $16$, and $32$.  We leave it to the motivated reader to explain these congruences modulo powers of $2$.  
\end{quotation}
Our primary goal in this note is to address Broudy and Lovejoy's invitation to consider such arithmetic properties modulo $8$, $16$, and $32$.  In particular, we prove (via mathematical induction) that an infinite family of Ramanujan-like divisibility properties is satisfied by $S(n)$ modulo $16$ by utilizing an internal congruence modulo $16$ for $S(n)$.  (This is similar in spirit to the work accomplished by Broudy and Lovejoy modulo $5$.) We also prove two internal congruences modulo 8 which are satisfied by $S(n).$  In our closing section, we share comments related to internal congruences modulo 32 which appear to be satisfied by $S(n)$.

All of the proof techniques used herein are elementary, relying on classical $q$-series identities and generating function manipulations as well as the parameterization work of Alaca, Alaca, and Williams \cite{AAW2006}.  


\section{Preliminaries}

As was noted above, 
\begin{align}
\sum_{n=0}^{\infty} S(n)q^n & = \frac{1}{(q;q^6)_\infty(q^5;q^6)_\infty(q^6;q^6)_\infty} \nonumber \\
& = \frac{(q^2;q^6)_\infty(q^3;q^6)_\infty(q^4;q^6)_\infty}{(q;q)_\infty} \nonumber \\
& = \frac{(q^2;q^6)_\infty(q^4;q^6)_\infty(q^6;q^6)_\infty\cdot (q^3;q^6)_\infty(q^6;q^6)_\infty}{(q;q)_\infty(q^6;q^6)_\infty^2} \nonumber \\
& = \frac{f_2f_3}{f_1f_6^2}, \label{gf}
\end{align}
where we use the following shorthand notation for $q$-Pochhammer symbols:
\begin{align*}
    f_r := (q^r;q^r)_\infty.
\end{align*}

In order to prove our results, we will require several elementary generating function dissection tools.  We begin by sharing a number of $2$-dissection results that will be helpful.  

\begin{lemma}
We have
\begin{align*}
{f_1^2} & = \frac{f_2f_8^5}{f_4^2f_{16}^2} - 2q \frac{f_2f_{16}^2}{f_8}, \\
\frac{1}{f_1^2} & = \frac{f_8^5}{f_2^5f_{16}^2} + 2q \frac{f_4^2f_{16}^2}{f_2^5f_8}, \\
\frac{1}{f_1^4} & = \frac{f_4^{14}}{f_2^{14}f_{8}^4} + 4q \frac{f_4^2f_{8}^4}{f_2^{10}}, \\
\frac{1}{f_1^8} & = \frac{f_4^{28}}{f_2^{28}f_{8}^8} + 8q \frac{f_4^{16}}{f_2^{24}} +16q^2\frac{f_4^4f_8^8}{f_2^{20}}.
\end{align*}
\label{L4}
\end{lemma}

\begin{proof}
The first two identities follow from dissection formulas of Ramanujan which appear in the work of Berndt \cite[p.~40]{Berndt}. The third identity can be found in the work of Brietzke, da Silva, and Sellers \cite[Equation (18)]{BSS}. The fourth identity follows by squaring the third identity.  
\end{proof}

\begin{lemma}
We have
\begin{equation}
\frac{f_3}{f_1} = \frac{f_4f_6f_{16}f_{24}^2}{f_2^2f_8f_{12}f_{48}} + q \frac{f_6f_8^2f_{48}}{f_2^2f_{16}f_{24}}.   
\label{eq3}
\end{equation}
\end{lemma}

\begin{proof}
See \cite[Equation (30.10.3)]{H}. 
\end{proof}

\begin{lemma}
We have
\begin{align}
\frac{f_3^3}{f_1} & = \frac{f_4^3f_6^2}{f_2^2f_{12}} + q \frac{f_{12}^3}{f_4},   
\label{eq5}  \\
\frac{f_1}{f_3^3} & = \frac{f_2f_4^2f_{12}^2}{f_6^7} - q \frac{f_2^3f_{12}^6}{f_4^2f_6^9}.   
\label{eq2.6.1}
\end{align}
\label{L2}
\end{lemma}

\begin{proof}
The first identity is equivalent to \cite[Equation (22.7.5)]{H}. Replacing $q$ by $-q$ in \eqref{eq5} and using the fact 
$$(-q;-q)_\infty = \frac{f_2^3}{f_1f_4},$$
we obtain \eqref{eq2.6.1}.
\end{proof}

\begin{lemma}
We have
\begin{align*}
\frac{1}{f_1f_3} & = \frac{f_8^2f_{12}^5}{f_2^2f_4f_6^4f_{24}^2} + q \frac{f_4^5f_{24}^2}{f_2^4f_6^2f_8^2f_{12}},  \\
f_1f_3 & = \frac{f_2f_8^2f_{12}^4}{f_4^2f_6f_{24}^2} - q \frac{f_4^4f_6f_{24}^2}{f_2f_8^2f_{12}^2}. 
\end{align*}
\label{L3}
\end{lemma}

\begin{proof}
Both of these identities can be found in Hirschhorn \cite{H}; the first is Equation (30.12.3) while the second is Equation (30.12.1).
\end{proof}

Another ingredient that plays a significant role in our work involves the following parameterization relations for eta functions due to Alaca, Alaca, and Williams \cite{AAW2006}.

\begin{lemma}\label{le:AAW-para}
    We have
    \begin{align*}
		f_1&=s^{\frac{1}{2}}t^{\frac{1}{24}}(1-2qt)^{\frac{1}{2}}(1+qt)^{\frac{1}{8}}(1+2qt)^{\frac{1}{6}}(1+4qt)^{\frac{1}{8}},\\
		f_2&=s^{\frac{1}{2}}t^{\frac{1}{12}}(1-2qt)^{\frac{1}{4}}(1+qt)^{\frac{1}{4}}(1+2qt)^{\frac{1}{12}}(1+4qt)^{\frac{1}{4}},\\
		f_3&=s^{\frac{1}{2}}t^{\frac{1}{8}}(1-2qt)^{\frac{1}{6}}(1+qt)^{\frac{1}{24}}(1+2qt)^{\frac{1}{2}}(1+4qt)^{\frac{1}{24}},\\
		f_4&=s^{\frac{1}{2}}t^{\frac{1}{6}}(1-2qt)^{\frac{1}{8}}(1+qt)^{\frac{1}{2}}(1+2qt)^{\frac{1}{24}}(1+4qt)^{\frac{1}{8}},\\
		f_6&=s^{\frac{1}{2}}t^{\frac{1}{4}}(1-2qt)^{\frac{1}{12}}(1+qt)^{\frac{1}{12}}(1+2qt)^{\frac{1}{4}}(1+4qt)^{\frac{1}{12}},\\
		f_{12}&=s^{\frac{1}{2}}t^{\frac{1}{2}}(1-2qt)^{\frac{1}{24}}(1+qt)^{\frac{1}{6}}(1+2qt)^{\frac{1}{8}}(1+4qt)^{\frac{1}{24}},
	\end{align*}
	where
	\begin{align*}
		s=s(q)&:=\frac{\varphi^3(q^3)}{\varphi(q)},\\
		t=t(q)&:=\frac{\varphi^2(q)-\varphi^2(q^3)}{4q\varphi^2(q^3)},
	\end{align*}
    and $$\varphi(q) := 1+2\sum_{n=1}^\infty q^{n^2}$$
    is the well-known theta function of Ramanujan.
\end{lemma}


\section{Dissections of the Generating Function for $S(n)$}

In order to complete our proofs below, we will need expressions for various $2$-, $4$-, and $8$-dissections of the generating function for $S(n)$.  For the sake of completeness, we provide all of these below.  
We begin by proving the $2$-dissection for $S(n)$.

\begin{theorem}
We have
\begin{align}
\sum_{n=0}^{\infty} S(2n)q^n & = \frac{f_2f_8f_{12}^2}{f_1f_3f_4f_6f_{24}}, \label{eq1} \\
\sum_{n=0}^{\infty} S(2n+1)q^n & = \frac{f_4^2f_{24}}{f_1f_3f_8f_{12}}. \label{eq2}
\end{align}
\end{theorem}

\begin{proof}
Using \eqref{eq3}, we can rewrite \eqref{gf} as
$$\sum_{n=0}^{\infty} S(n)q^n = \frac{f_4f_{16}f_{24}^2}{f_2f_6f_8f_{12}f_{48}} + q\frac{f_8^2f_{48}}{f_2f_6f_{16}f_{24}}.$$
Extracting the even parts on both sides of this identity and replacing $q^2$ by $q$ we obtain \eqref{eq1}. Finally, extracting the odd parts on both sides, dividing by $q$, and replacing $q^2$ by $q$ yield \eqref{eq2}.
\end{proof}

\begin{theorem}
We have
\begin{align}
\sum_{n=0}^{\infty} S(4n)q^n & = \frac{f_4^3f_6^7}{f_1f_2^2f_3^5f_{12}^3}
\label{4diss}, \\
\sum_{n=0}^{\infty} S(4n+1)q^n & = \frac{f_2f_4f_{6}^4}{f_1^2f_3^4f_{12}}, \label{4diss_1} \\
\sum_{n=0}^{\infty} S(4n+2)q^n & = \frac{f_2^4f_6f_{12}}{f_1^3f_3^3f_4}, \label{4diss_2} 
\\
\sum_{n=0}^{\infty} S(4n+3)q^n & = \frac{f_2^7f_{12}^3}{f_1^4f_3^2f_4^3f_6^2}. \label{4diss_3}
\end{align}
\end{theorem}

\begin{proof}
We begin by using Lemma \ref{L3} to $2$-dissect \eqref{eq2}:
\begin{align*}
\sum_{n=0}^{\infty} S(4n+1)q^{2n} & = \frac{f_4f_8f_{12}^4}{f_2^2f_6^4f_{24}},  \\
\sum_{n=0}^{\infty} S(4n+3)q^{2n+1} & = q\frac{f_4^7f_{24}^3}{f_2^4f_6^2f_8^3f_{12}^2}.
\end{align*}
After dividing the last expression by $q$ and replacing $q^2$ by $q$ in both identities, we obtain \eqref{4diss_1} and \eqref{4diss_3}. 

Using Lemma \ref{L3} again we can $2$-dissect \eqref{eq1}, which yields
\begin{align*}
\sum_{n=0}^{\infty} S(4n)q^{2n} & = \frac{f_8^3f_{12}^7}{f_2f_4^2f_6^5f_{24}^3},  \\
\sum_{n=0}^{\infty} S(4n+2)q^{2n+1} & = q\frac{f_4^4f_{12}f_{24}}{f_2^3f_6^3f_8}.
\end{align*}
As before, after dividing the last expression by $q$ and replacing $q^2$ by $q$ in both identities, we obtain \eqref{4diss} and \eqref{4diss_2}.
\end{proof}

\begin{theorem}
We have
\begin{align}
\sum_{n=0}^{\infty} S(8n)q^n & = \frac{f_2^2f_4^2f_{6}^{16}}{f_1^4f_3^{11}f_{12}^6}+4q^2\frac{f_2^8f_{12}^6}{f_1^6f_3^5f_4^2f_6^2},
\label{8diss} \\
\sum_{n=0}^{\infty} S(8n+1)q^n & = \frac{f_2f_4^5f_{6}^{13}}{f_1^4f_3^{10}f_8^2f_{12}^4}+8q^2\frac{f_2^3f_6f_8^2f_{12}^4}{f_1^4f_3^6f_4}, \label{8diss_1} \\
\sum_{n=0}^{\infty} S(8n+2)q^n & = \frac{f_2^{15}f_{6}^{3}}{f_1^9f_3^{6}f_4^4}-4q\frac{f_4^4f_6^7}{f_1^3f_2f_3^8}, \label{8diss_2} 
\\
\sum_{n=0}^{\infty} S(8n+3)q^n & = \frac{f_2^{11}f_6^3f_{12}^{5}}{f_1^7f_3^7f_4^4f_{24}^2}+8q^2\frac{f_4^4f_6^5f_{24}^2}{f_1^3f_2f_3^7f_{12}}, \label{8diss_3} \\
\sum_{n=0}^{\infty} S(8n+4)q^n & = \frac{f_2^8f_{6}^{10}}{f_1^6f_3^{9}f_4^2f_{12}^2}+4q\frac{f_2^2f_4^2f_6^4f_{12}^2}{f_1^4f_3^7}, \label{8diss_4} \\
\sum_{n=0}^{\infty} S(8n+5)q^n & = 2\frac{f_2^3f_6^{13}f_{8}^2}{f_1^4f_3^{10}f_4f_{12}^4}+4q\frac{f_2f_4^5f_6f_{12}^4}{f_1^4f_3^6f_8^2}, \label{8diss_5} \\
\sum_{n=0}^{\infty} S(8n+6)q^n & = 4\frac{f_2^3f_4^4f_{6}^{3}}{f_1^5f_3^{6}}-\frac{f_2^{11}f_6^7}{f_1^7f_3^8f_4^4}, \label{8diss_6} \\
\sum_{n=0}^{\infty} S(8n+7)q^n & = 4\frac{f_4^4f_6^3f_{12}^{5}}{f_1^3f_2f_3^{7}f_{24}^2}+2q\frac{f_2^{11}f_6^5f_{24}^2}{f_1^7f_3^7f_4^4f_{12}}. \label{8diss_7}
\end{align}
\end{theorem}

\begin{proof}
Thanks to the identities in Lemma \ref{L4} we can rewrite \eqref{4diss_1} as
\begin{align*}
\sum_{n=0}^{\infty} S(4n+1)q^n = \frac{f_2f_4f_6^4}{f_{12}} \left( \frac{f_8^5}{f_2^5f_{16}^2} + 2q \frac{f_4^2f_{16}^2}{f_2^5f_8} \right) \left( \frac{f_{12}^{14}}{f_6^{14}f_{24}^4} + 4q^3 \frac{f_{12}^2f_{24}^4}{f_6^{10}}\right).
\end{align*}
Extracting the even and the odd parts of this identity we obtain
\begin{align*}
\sum_{n=0}^{\infty} S(8n+1)q^{2n} & = \frac{f_4f_8^5f_{12}^{13}}{f_2^4f_6^{10}f_{16}^2f_{24}^4} +8q^4\frac{f_4^3f_{12}f_{16}^2f_{24}^4}{f_2^4f_6^6f_8}, \\    
\sum_{n=0}^{\infty} S(8n+5)q^{2n+1} & = 2q\frac{f_4^3f_{12}^{13}f_{16}^2}{f_2^4f_6^{10}f_8f_{24}^4} +4q^3\frac{f_4f_8^5f_{12}f_{24}^4}{f_2^4f_6^6f_{16}^2},
\end{align*}
which after simplifications yield \eqref{8diss_1} and \eqref{8diss_5}, respectively. 

Now we use Lemma \ref{L4} to rewrite \eqref{4diss_3} as
\begin{align*}
\sum_{n=0}^{\infty} S(4n+3)q^n = \frac{f_2^7f_{12}^3}{f_4^3f_{6}^2} \left( \frac{f_{24}^5}{f_6^5f_{48}^2} + 2q^3 \frac{f_{12}^2f_{48}^2}{f_6^5f_{24}} \right) \left( \frac{f_{4}^{14}}{f_2^{14}f_{8}^4} + 4q \frac{f_{4}^2f_{8}^4}{f_2^{10}}\right).
\end{align*}
Extracting the even and the odd parts of this identity we obtain
\begin{align*}
\sum_{n=0}^{\infty} S(8n+3)q^{2n} & = \frac{f_4^{11}f_{12}^{3}f_{24}^5}{f_2^7f_6^{7}f_{8}^4f_{48}^2} +8q^4\frac{f_8^4f_{12}^5f_{48}^2}{f_2^3f_4f_6^7f_{24}}, \\    
\sum_{n=0}^{\infty} S(8n+7)q^{2n+1} & = 4q\frac{f_8^4f_{12}^{3}f_{24}^5}{f_2^3f_4f_6^{7}f_{48}^2} +2q^3\frac{f_4^{11}f_{12}^5f_{48}^2}{f_2^7f_6^7f_8^4f_{24}},
\end{align*}
which after simplifications yield \eqref{8diss_3} and \eqref{8diss_7}, respectively.

In order to prove \eqref{8diss} and \eqref{8diss_4}, we make use of Lemma \ref{L4} and Lemma \ref{L3} to rewrite \eqref{4diss} in the following form
\begin{align*}
\sum_{n=0}^{\infty} S(4n)q^n = \frac{f_4^3f_{6}^7}{f_2^2f_{12}^3} \left( \frac{f_8^2f_{12}^5}{f_2^2f_4f_6^4f_{24}^2} + q \frac{f_4^5f_{24}^6}{f_2^4f_6^{2}f_{8}^2f_{12}} \right) \left( \frac{f_{12}^{14}}{f_6^{14}f_{24}^4} + 4q^3 \frac{f_{12}^2f_{24}^4}{f_6^{10}}\right).
\end{align*}
Extracting the even and the odd parts of this identity we obtain
\begin{align*}
\sum_{n=0}^{\infty} S(8n)q^{2n} & = \frac{f_4^{2}f_{8}^{2}f_{12}^{16}}{f_2^4f_6^{11}f_{24}^6} +4q^4\frac{f_4^8f_{24}^6}{f_2^6f_6^5f_8^2f_{12}^2}, \\    
\sum_{n=0}^{\infty} S(8n+4)q^{2n+1} & = q\frac{f_4^8f_{12}^{10}}{f_2^6f_6^{9}f_8^2f_{24}^2} +4q^3\frac{f_4^{2}f_8^2f_{12}^4f_{24}^2}{f_2^4f_6^7},
\end{align*}
which after simplifications yield \eqref{8diss} and \eqref{8diss_4}, respectively.

Finally we prove \eqref{8diss_2} and \eqref{8diss_6}. To this end, we make use of Lemma \ref{L4} and Lemma \ref{L2} to rewrite \eqref{4diss_2} in the following form
\begin{align*}
\sum_{n=0}^{\infty} S(4n+2)q^n = \frac{f_2^4f_{6}f_{12}}{f_4} \left( \frac{f_2f_4^2f_{12}^2}{f_6^7} - q \frac{f_2^3f_{12}^6}{f_4^2f_6^9} \right) \left( \frac{f_{4}^{14}}{f_2^{14}f_{8}^4} + 4q \frac{f_{4}^2f_{8}^4}{f_2^{10}}\right).
\end{align*}
Extracting the even and the odd parts of this identity we obtain
\begin{align*}
\sum_{n=0}^{\infty} S(8n+2)q^{2n} & = \frac{f_4^{15}f_{12}^{3}}{f_2^9f_6^{6}f_{8}^4} -4q^2\frac{f_8^4f_{12}^7}{f_2^3f_4f_6^8}, \\ 
\sum_{n=0}^{\infty} S(8n+6)q^{2n+1} & = 4q\frac{f_4^3f_8^4f_{12}^{3}}{f_2^5f_6^{6}} -q\frac{f_4^{11}f_{12}^7}{f_2^7f_6^8f_8^4},
\end{align*}
which after simplifications yield \eqref{8diss_2} and \eqref{8diss_6}, respectively.
\end{proof}


\section{An Infinite Family of Ramanujan-like Divisibility Properties Modulo $16$}
We now proceed to a proof of an infinite family of Ramanujan-like congruences modulo $16$.  We begin by explicitly proving the first few cases of this infinite family of results.   

\begin{theorem}
 \label{c1}
For all $n\geq 0$,
\begin{align}
S(32n+31) & \equiv 0 \pmod{16}, \label{eq22.6.1}  \\
S(128n+123) & \equiv 0 \pmod{16},  \label{eq22.6.2} \\
S(512n+491) & \equiv 0 \pmod{16}.  \label{eq22.6.3}
\end{align}
\end{theorem}

\begin{proof}
Thanks to \eqref{8diss_7}, we know
\begin{align*}
\sum_{n=0}^{\infty} S(8n+7)q^{n} & = 2q\frac{f_2^{11}f_6^5f_{24}^2}{f_1^7f_3^7f_4^4f_{12}} + 4\frac{f_4^4f_6^3f_{12}^5}{f_1^3f_2f_3^7f_{24}^2} \\
& \equiv 2qf_1f_3 \frac{f_2^7f_6f_{24}^2}{f_4^4f_{12}} + 4f_1f_3\frac{f_4^4f_{12}^5}{f_2^3f_6f_{24}^2} \pmod{16}
\end{align*}
using the elementary facts $2f_{m}^{8k} \equiv 2f_{2m}^{4k} \pmod{16}$ and $4f_{m}^{4k} \equiv 4f_{2m}^{2k} \pmod{16}$.
We now utilize Lemma \ref{L3} to extract the odd parts on both sides of the above congruence to obtain
$$\sum_{n=0}^{\infty} S(16n+15)q^{2n+1} \equiv 2q\frac{f_2^8f_8^2f_{12}^3}{f_4^6} - 4q\frac{f_4^8f_{12}^3f_{24}}{f_2^4f_8^2} \pmod{16}.$$
Thus, after dividing by $q$ and replacing $q^2$ by $q$ we get
\begin{align*}
\sum_{n=0}^{\infty} S(16n+15)q^{n} & \equiv 2\frac{f_1^8f_4^2f_6^3}{f_2^6} - 4 \frac{f_2^8f_6^3f_{12}}{f_1^4f_4^2} \\
& \equiv 2 \frac{f_4^2f_6^3}{f_2^2} - 4f_2^2f_6^3f_{12} \pmod{16},
\end{align*}
from which \eqref{eq22.6.1} follows since the right-hand side of the last expression is a function of $q^2$.

Next we prove \eqref{eq22.6.2}. We begin by rewriting \eqref{8diss_3} in the following form:
\begin{equation*}
\sum_{n=0}^{\infty} S(8n+3)q^n \equiv \frac{1}{f_1^8}\frac{f_1}{f_3^3}\frac{1}{f_3^4}\frac{f_2^{11}f_6^3f_{12}^5}{f_4^4f_{24}^2} + 8q^2 \frac{1}{f_1f_3}\frac{f_4^4f_6^2f_{12}^3}{f_2^2} \pmod{16}.    
\end{equation*}
Using Lemmas \ref{L4}, \ref{L2}, and \ref{L3}, we can extract the odd parts on both sides of this congruence to obtain
\begin{align*}
\sum_{n=0}^{\infty} S(16n+11)q^{2n+1} & \equiv 8q^3 \frac{f_4^9f_{12}^2f_{24}^2}{f_2^6f_8^2} + 
4q^3\frac{f_4^{26}f_{12}^9f_{24}^2}{f_2^{16}f_6^{14}f_8^8} \nonumber \\ & \ \ \ \  - q \frac{f_4^{22}f_{12}^{25}}{f_2^{14}f_6^{20}f_8^8f_{24}^6} + 8q \frac{f_4^{14}f_{12}^{21}}{f_2^{12}f_6^{18}f_{24}^6} \pmod{16},  
\end{align*}
from which, after dividing by $q$ and replacing $q^2$ by $q$, we have
\begin{align}
\sum_{n=0}^{\infty} S(16n+11)q^{n} & \equiv 8q \frac{f_2^9f_{6}^2f_{12}^2}{f_1^6f_4^2} + 
4q\frac{f_2^{26}f_{6}^9f_{12}^2}{f_1^{16}f_3^{14}f_4^8} \nonumber \\ & \ \ \ \  - \frac{f_2^{22}f_{6}^{25}}{f_1^{14}f_3^{20}f_4^8f_{12}^6} + 8 \frac{f_2^{14}f_{6}^{21}}{f_1^{12}f_3^{18}f_{12}^6} \nonumber \\ 
& \equiv 8qf_2^2f_6^6 +4qf_2^2f_3^2f_6^5 - \frac{f_1^2f_6^{17}}{f_2^2f_3^4f_{12}^6} + 8f_2^8 \pmod{16}. \label{eq26.6.1}
\end{align}
Now we use Lemma \ref{L4} to extract the odd parts on both sides of the last congruence, which yields 
\begin{align*}
\sum_{n=0}^{\infty} S(32n+27)q^{2n+1} & \equiv 8q f_2^2f_{6}^6 + 
4q\frac{f_2^{2}f_{6}^6f_{24}^5}{f_{12}^{2}f_{48}^2} \nonumber \\ & \ \ \ \  - 4q^3\frac{f_6^{7}f_{8}^{5}f_{24}^4}{f_2f_{4}^2f_{12}^{4}f_{16}^2} + 2q\frac{f_6^{3}f_{12}^{8}f_{16}^2}{f_2f_8f_{24}^4} \pmod{16}.
\end{align*}
Dividing this congruence by $q$ and replacing $q^2$ by $q$, we have
\begin{align*}
\sum_{n=0}^{\infty} S(32n+27)q^{n} & \equiv 8 f_1^2f_{3}^6 + 
4\frac{f_1^{2}f_{3}^6f_{12}^5}{f_{6}^{2}f_{24}^2} - 4q\frac{f_3^{7}f_{4}^{5}f_{12}^4}{f_1f_{2}^2f_{6}^{4}f_{8}^2} + 2\frac{f_3^{3}f_{6}^{8}f_{8}^2}{f_1f_4f_{12}^4} \\
& \equiv 8f_2f_6^3 + 4f_1^2f_3^2f_{12} - 4q\frac{f_3^3}{f_1} \frac{f_4f_6^6}{f_2^2} + 2\frac{f_3^3}{f_1}\frac{f_6^8f_8^2}{f_4f_{12}^4} \pmod{16}.
\end{align*}
Now we employ Lemma \ref{L4} and Lemma \ref{L2} to obtain
\begin{align*}
\sum_{n=0}^{\infty} S(64n+59)q^{2n+1} & \equiv 8q^3 \frac{f_2f_6f_8^5f_{12}f_{48}^2}{f_4^2f_{16}^2f_{24}} + 
8q\frac{f_2f_6f_{16}^2f_{24}^5}{f_{8}f_{12}f_{48}^2} \nonumber \\ & \ \ \ \  - 4q\frac{f_4^{4}f_{6}^{8}}{f_{2}^4f_{12}} + 2q\frac{f_6^{8}f_{8}^{2}}{f_4^2f_{12}} \pmod{16},
\end{align*}
from which we get
\begin{align*}
\sum_{n=0}^{\infty} S(64n+59)q^{n} & \equiv 8q f_1f_3f_6^7 + 
8f_1f_3f_2^6f_6  - 4f_2^2f_6^3 + 2\frac{f_4^{2}f_{6}^{3}}{f_2^2} \pmod{16}.
\end{align*}
Finally, using Lemma \ref{L3}, we can extract the odd parts of the congruence above to obtain
\begin{align*}
\sum_{n=0}^{\infty} S(128n+123)q^{2n+1} & \equiv 8q \frac{f_2f_6^6f_8^2f_{12}^4}{f_4^2f_{24}^2} + 
8q\frac{f_2^5f_4^4f_6^2f_{24}^2}{f_{8}^2f_{12}^2} \\
& \equiv 8qf_2^5f_6^6 + 8qf_2^5f_6^6 \equiv 0 \pmod{16},
\end{align*}
which completes the proof of \eqref{eq22.6.2}.

In order to prove \eqref{eq22.6.3}, we use Lemma \ref{L4} to extract the even parts on both sides of \eqref{eq26.6.1} which yields
\begin{align*}
\sum_{n=0}^{\infty} S(32n+11) q^n & \equiv 8q^2f_2f_6^3f_{24}^3 + 8f_2^4 - \frac{f_3^3f_4^5f_6^8}{f_1f_2^2f_8^2f_{12}^4} + 8q^2\frac{f_3f_4^3f_6^7}{f_1} \pmod{16}.
\end{align*}
We now use \eqref{eq3} and \eqref{eq5} to extract the odd parts on both sides of the last congruence to obtain
\begin{align*}
\sum_{n=0}^{\infty} S(64n+43) q^{2n+1} & \equiv -q\frac{f_4^4f_6^8}{f_2^2f_8^2f_{12}} + 8q^3\frac{f_4^3f_6^8f_8^2f_{48}}{f_2^2f_{16}f_{24}} \pmod{16},
\end{align*}
which after dividing by $q$ and replacing $q^2$ by $q$ yields
\begin{align}
\sum_{n=0}^{\infty} S(64n+43) q^{n} & \equiv -\frac{f_2^4f_3^8}{f_1^2f_4^2f_{6}} + 8q\frac{f_2^3f_3^8f_4^2f_{24}}{f_1^2f_{8}f_{12}} \notag\\
& \equiv -\frac{f_2^4f_3^8}{f_1^2f_4^2f_{6}} + 8qf_2^2f_6^6 \pmod{16}.\label{64n43mod16}
\end{align}
Using Lemma \ref{L4} we obtain
\begin{align*}
\sum_{n=0}^{\infty} S(128n+107) q^{2n+1} & \equiv 8q^3\frac{f_8^5f_{12}^8}{f_2f_4^2f_6f_{16}^2} - 2q\frac{f_{12}^{20}f_{16}^2}{f_2f_6^5f_{8}f_{24}^8} +8qf_2^2f_6^6 \pmod{16},
\end{align*}
from which it follows that
\begin{align*}
\sum_{n=0}^{\infty} S(128n+107) q^{n} & \equiv 8q\frac{f_4^5f_{6}^8}{f_1f_2^2f_3f_{8}^2} - 2\frac{f_3^3f_8^2}{f_1f_4} +8f_2f_6^3 \pmod{16}.
\end{align*}
Finally, we use \eqref{eq3} and \eqref{eq5} to obtain
\begin{align*}
\sum_{n=0}^{\infty} S(256n+235) q^{2n+1} & \equiv 8q\frac{f_4^4f_{6}^4f_{12}^5}{f_2^4f_{24}^2} - 2q\frac{f_8^2f_{12}^3}{f_4^2} \pmod{16},
\end{align*}
which, after division by $q$ and replacement of $q^2$ by $q$, yields
\begin{align*}
\sum_{n=0}^{\infty} S(256n+235) q^{n} & \equiv 8f_2^2f_6^3 - 2\frac{f_4^2f_{6}^3}{f_2^2} \pmod{16}.
\end{align*}
This proves \eqref{eq22.6.3} since the right-hand side of the last congruence is a function of $q^2$. 
\end{proof}

Next, we prove an internal congruence which will provide the machinery necessary to prove an infinite family of congruences for $S(n)$ modulo $16$.  
\begin{theorem}
  \label{internal_16}
For all $n\geq 0$,  
$$S\big(256n+171\big) = S\big(4(64n+43)-1\big) \equiv S\big(64n+43\big) \pmod{16}.$$
\end{theorem}
\begin{proof}  In \eqref{64n43mod16}, we noted that 
$$
\sum_{n=0}^{\infty} S(64n+43) q^{n} 
\equiv -\frac{f_2^4f_3^8}{f_1^2f_4^2f_{6}} + 8qf_2^2f_6^6 \pmod{16}.
$$
We next need to determine a similar congruence modulo 16 for the generating function dissection for $S(256n+171)$.
Using \eqref{64n43mod16} and the same kinds of dissection tools that have already been demonstrated, we can easily show that 
\begin{align*}
\sum_{n=0}^{\infty} S(128n+43) q^{n} 
&\equiv -\frac{f_4^5f_6^{20}}{f_1f_2^2f_3^5f_8^2f_{12}^8} \\
&\equiv -\frac{f_4^5f_6^{20}}{f_2^2f_8^2f_{12}^8}\left( \frac{1}{f_1f_3}\right)\left( \frac{1}{f_3^4}\right)
\pmod{16}.
\end{align*}
Using Lemmas \ref{L4} and \ref{L3}, we can then show that 
$$
\sum_{n=0}^{\infty} S(256n+171) q^{n} 
\equiv -\frac{f_2^{10} f_3^4 f_6^5}{f_1^6 f_4^4 f_{12}^2} -4q \frac{ f_2^4 f_3^6 f_{12}^2}{f_1^4 f_6} \pmod{16}.
$$
Therefore, in order to prove this theorem, we want to prove 
	\begin{align*}
		-\frac{f_2^4 f_3^8}{f_1^2 f_4^2 f_6} + 8q f_4 f_6^6 \equiv -\frac{f_2^{10} f_3^4 f_6^5}{f_1^6 f_4^4 f_{12}^2} -4q \frac{ f_2^4 f_3^6 f_{12}^2}{f_1^4 f_6} \pmod{16}.
	\end{align*}
To do so, we begin by noting that
	\begin{align*}
		\frac{f_2 f_6^2}{f_1^2 f_{12}} \equiv 1 \pmod{2}.
	\end{align*}
	Thus,
	\begin{align*}
		8q f_4 f_6^6 \equiv 8q\frac{f_2 f_4 f_6^8}{f_1^2 f_{12}} \pmod{16}.
	\end{align*}
	Hence,  it is sufficient to prove that $L(q)\equiv 0 \pmod{16} $ where 
	\begin{align}\label{eq:mod-16-to-be-proved-new}
		L(q):= -\frac{f_2^4 f_3^8}{f_1^2 f_4^2 f_6} + 8q \frac{f_2 f_4 f_6^8}{f_1^2 f_{12}} + \frac{f_2^{10} f_3^4 f_6^5}{f_1^6 f_4^4 f_{12}^2} + 4q \frac{ f_2^4 f_3^6 f_{12}^2}{f_1^4 f_6}.
	\end{align}
	We then make use of the Alaca--Alaca--Williams parameterization. Substituting the relations in Lemma \ref{le:AAW-para} into $L(q)$, we have
	\begin{align*}
		L(q)\cdot \frac{f_4^2 f_6}{f_3^2} = 16 q s^4 t^2 (1 + q t)^3 (1 + 2 q t) (1 + 4 q t).
	\end{align*}
	Since the series $f_4^2 f_6/f_3^2$ is invertible in the ring $\mathbb{Z}/16\mathbb{Z}[[q]]$, while the right-hand side of the above vanishes in the same ring, we conclude that
	\begin{align*}
		L(q) \equiv 0 \pmod{16},
	\end{align*}
	as required.
\end{proof}

Theorem \ref{internal_16} now allows us to efficiently prove the following infinite family of congruences modulo $16$:  

\begin{theorem}
\label{congs_mod16}
For all $\alpha\geq 0$ and $n\geq 0$,
$$
S\left(2^{5+2\alpha}n+\left(2^{5+2\alpha}-\frac{2^{2+2\alpha}-1}{3}\right)\right)\equiv 0 \pmod{16}.
$$
\end{theorem}

\begin{proof}
We prove this theorem by induction on $\alpha$.  Note that the cases $\alpha=0, 1, 2$ have already been proved in Theorem \ref{c1} above.  We now use the case $\alpha=2$ as our base case in a proof by induction.   
Thus, we assume that, for some $\alpha\geq 2$ and all $n\geq 0$, 
$$
S\left(2^{5+2\alpha}n+\left(2^{5+2\alpha}-\frac{2^{2+2\alpha}-1}{3}\right)\right)\equiv 0 \pmod{16}.
$$
We then want to prove that 
$$
S\left(2^{5+2(\alpha+1)}n+\left(2^{5+2(\alpha+1)}-\frac{2^{2+2(\alpha+1)}-1}{3}\right)\right)\equiv 0 \pmod{16}.
$$
It is important to note that, for all $\alpha\geq 2$, 
$$
2^{5+2\alpha}n+\left(2^{5+2\alpha}-\frac{2^{2+2\alpha}-1}{3}\right) \equiv 43 \pmod{64}.
$$
This will allow the induction step below to go through via Theorem \ref{internal_16}.
Indeed, 
\begin{align*}
&\!\!\!\!\!\!\!\!\!\!\!\! S\left(2^{5+2(\alpha+1)}n+\left(2^{5+2(\alpha+1)}-\frac{2^{2+2(\alpha+1)}-1}{3}\right)\right) \\
&= S\left(2^{5+2(\alpha+1)}n+\left(2^{5+2(\alpha+1)}-\frac{2^{2+2(\alpha+1)}-4}{3}\right)-1\right) \\
&= 
S\left(4\left(2^{5+2\alpha}n+\left(2^{5+2\alpha}-\frac{2^{2+2\alpha}-1}{3}\right)\right) - 1\right) \\
&\equiv 
S\left(2^{5+2\alpha}n+\left(2^{5+2\alpha}-\frac{2^{2+2\alpha}-1}{3}\right)\right) \\
&\equiv 
0 \pmod{16},  
\end{align*}
where Theorem \ref{internal_16} is utilized in the penultimate congruence above while the induction hypothesis is used in the last. The result follows.
\end{proof}


\section{Internal Congruences Modulo 8}

While computationally searching for the result in Theorem \ref{internal_16}, we discovered other internal congruences satisfied by $S(n)$.  In this section, we focus our attention on internal congruences modulo $8$.  
\begin{theorem}
  \label{internal_8}
For all $n\geq 0$,  
\begin{align}
    S\big(64n+43\big) &= S\big(4(16n+11)-1\big) \equiv S\big(16n+11\big) \pmod{8}, \label{internal_8.1} \\
    S\big(64n+59\big) &= S\big(4(16n+15)-1\big) \equiv S\big(16n+15\big) \pmod{8}. \label{internal_8.2}
\end{align}
\end{theorem}

\begin{proof}
It follows from \eqref{8diss_3} that
\begin{align*}
\sum_{n=0}^{\infty} S(8n+3)q^n & \equiv \frac{f_2^{11}f_6^3f_{12}^5}{f_1^7f_3^7f_4^4f_{24}^2} \\
& \equiv f_1f_3 \frac{f_{12}^5}{f_2f_6f_{24}^2} \pmod{8}.
\end{align*}
We can use Lemma \ref{L3} to extract the odd parts on both sides of the last expression, which yields
\begin{align*}
\sum_{n=0}^{\infty} S(16n+11)q^{2n+1} & \equiv -q \frac{f_4^4f_{12}^3}{f_2^2f_8^2} \pmod{8}. 
\end{align*}
Dividing by $q$ and replacing $q^2$ by $q$, we obtain
\begin{align}
\sum_{n=0}^{\infty} S(16n+11)q^{n} & \equiv -\frac{f_2^4f_{6}^3}{f_1^2f_4^2} \pmod{8}. 
\label{eq1.6.1}
\end{align}
Using Lemma \ref{L4}, we can 2-dissect \eqref{eq1.6.1} to obtain
\begin{align*}
\sum_{n=0}^{\infty} S(32n+27)q^{2n+1} & \equiv -2q\frac{f_6^3f_{16}^2}{f_2f_8} \pmod{8},\\
\sum_{n=0}^{\infty} S(32n+11)q^{2n} & \equiv -\frac{f_6^3f_{8}^5}{f_2f_4^2f_{16}^2} \pmod{8},
\end{align*}
which, after simplifications, give us
\begin{align*}
\sum_{n=0}^{\infty} S(32n+27)q^{n} & \equiv -2\frac{f_3^3f_{8}^2}{f_1f_4} \pmod{8},\\
\sum_{n=0}^{\infty} S(32n+11)q^{n} & \equiv -\frac{f_3^3f_{4}^5}{f_1f_2^2f_{8}^2} \pmod{8}. 
\end{align*}
We need the odd parts of both identities above, which are obtained by employing Lemma \ref{L2}:
\begin{align}
\sum_{n=0}^{\infty} S(64n+59)q^{n} & \equiv 6 f_2^2f_6^3 \pmod{8},\label{eq1.6.2} \\
\sum_{n=0}^{\infty} S(64n+43)q^{n} & \equiv -\frac{f_2^4f_{6}^3}{f_1^2f_{4}^2} \pmod{8}. \label{eq1.6.3}
\end{align}
We see that \eqref{internal_8.1} follows from \eqref{eq1.6.1} and \eqref{eq1.6.3}.

In order to prove \eqref{internal_8.2} we use the elementary facts $2f_{m}^{4k} \equiv 2f_{2m}^{2k} \pmod{8}$ and $4f_{m}^{2k} \equiv 4f_{2m}^{k} \pmod{8}$ to rewrite  \eqref{8diss_7} as
\begin{equation*}
\sum_{n=0}^{\infty} S(8n+7)q^{n} \equiv 4\frac{1}{f_1f_3} f_2^6f_{12} +2qf_1f_3\frac{f_6f_{12}^3}{f_2} \pmod{8}.
\end{equation*}
Using Lemma \ref{L3}, we extract the odd parts on both sides of the congruence above to obtain
\begin{align}
\sum_{n=0}^{\infty} S(16n+15)q^{n} & \equiv 4\frac{f_1^2f_2^5f_{12}^2}{f_3^2f_4^2} + 2\frac{f_4^2f_{6}^7}{f_2^2f_{12}^2} \nonumber \\
& \equiv 6f_2^2f_6^3 \pmod{8}. \label{eq1.6.4}
\end{align}
Note that \eqref{internal_8.2} follows by comparing \eqref{eq1.6.2} and \eqref{eq1.6.4}.
\end{proof}


\section{Closing Thoughts}  

In the work above, we have accomplished our primary goal, which was to prove an infinite family of Ramanujan-like congruences satisfied by $S(n)$ modulo $16$.  With that said, we readily admit that our results above are not exhaustive; computational evidence indicates, for example, that there are Ramanujan-like congruences modulo $8$ which are satisfied by $S(n)$ which are not covered by our results above.  Moreover, our data support the following congruences modulo $32$:

\begin{conjecture}
  \label{internal_32}
For all $n\geq 0$,  
\begin{align*}
    S\big(256n+123\big) &= S\big(4(64n+31)-1\big) \equiv S\big(64n+31\big) \pmod{32},  \\
    S\big(256n+171\big) &= S\big(4(64n+43)-1\big) \equiv S\big(64n+43\big) \pmod{32},  \\
    S\big(256n+235\big) &= S\big(4(64n+59)-1\big) \equiv S\big(64n+59\big) \pmod{32},  \\
    S\big(256n+251\big) &= S\big(4(64n+63)-1\big) \equiv S\big(64n+63\big) \pmod{32}.  
\end{align*}
\end{conjecture}

We leave it to the interested reader to consider proving additional Ramanujan-like congruences which are satisfied by $S(n)$ given the existence of the results mentioned above.

\end{document}